\documentclass[12pt,reqno]{article}
\usepackage[usenames]{color}
\usepackage[colorlinks=true,linkcolor=webgreen,filecolor=webbrown,citecolor=webgreen]{hyperref}
\usepackage{amssymb,amsmath,amsfonts}
\usepackage{enumerate}

\definecolor{webgreen}{rgb}{0,.5,0}
\definecolor{webbrown}{rgb}{.6,0,0}
\newtheorem{theorem}{Theorem}

\newtheorem{corollary}[theorem]{Corollary}

\newtheorem{example}[theorem]{Example}

\newtheorem{remark}[theorem]{Remark}

\newenvironment{proof}[1][Proof]{\noindent\textbf{#1.} }{\ \rule{0.5em}{0.5em}}

\allowdisplaybreaks
\input{tcilatex}
\begin{document}

\begin{center}
\vskip1cm

{\LARGE \textbf{An identity for derivatives}}

\vspace{2cm}

{\large Ulrich Abel}\\[3mm]
\textit{Fachbereich MND}\\[0pt]
\textit{Technische Hochschule Mittelhessen}\\[0pt]
\textit{Wilhelm-Leuschner-Stra\ss e 13, 61169 Friedberg, }\\[0pt]
\textit{Germany}\\[0pt]
\href{mailto:Ulrich.Abel@mnd.thm.de}{\texttt{Ulrich.Abel@mnd.thm.de}}
\end{center}

\vspace{2cm}

{\large \textbf{Abstract.}}

\bigskip

We present a generalization of a formula of higher order derivatives and
give a short proof.

\bigskip

\textit{Mathematics Subject Classification (2010)}: 26A06, 
26A24. 

\emph{Keywords: }One-variable calculus, differentiation, higher derivatives.

\vspace{2cm}

\section{Introduction}

Recently, Baran e.a. \cite{Baran-ea-2015} published the interesting identity 
\begin{equation}
\frac{1}{n!}\sum_{k=0}^{n}\left( -1\right) ^{k}\binom{n}{k}g^{k}\left(
fg^{n-k}\right) ^{\left( n\right) }=f\left( g^{\prime }\right) ^{n}
\label{baran-identity}
\end{equation}%
$\left( n=0,1,2,\ldots \right) $, for sufficiently often differentiable
functions. This identity is valid in each complex commutative algebra $%
\mathcal{A}$ with unity and each derivation operator $D:\mathcal{%
A\rightarrow A}$ (a linear operator with the property $D\left( fg\right)
=gD\left( f\right) +fD\left( g\right) $) with the notation $D\left( f\right)
=f^{\prime }$ and $D^{\left( k\right) }\left( f\right) =f^{\left( k\right) }$%
. A proof of identity $\left( \ref{baran-identity}\right) $ is annonounced
to be given in the forthcoming paper \cite{Baran-ea-in preparation}. A first
version with $f=1$ was found by B. Mil\'{o}wka \cite{Milowka-2005}\cite%
{Milowka-2006} in 2005, for polynomials $g$. In 2012, P. Ozorka proved it
for arbitrary functions $g$.

Such identities play a key role in deriving Markov type inequalities. M.
Baran observed that the Mil\'{o}wka identity implies a lower estimate for
the $k$-th derivative of polynomials considered on planar A. Markov sets.

The purpose of this note is a short transparent proof of a generalization of
identity $\left( \ref{baran-identity}\right) $. We restrict ourselves to
ordinary differentiation of functions on real or complex domains. One
special instance is the symmetric form 
\begin{equation}
\sum_{k=0}^{n}\left( -1\right) ^{k}\binom{n}{k}\left( f_{1}g^{k}\right)
^{\left( p\right) }\left( f_{2}g^{n-k}\right) ^{\left( n-p\right) }=\left(
-1\right) ^{p}n!f_{1}f_{2}\left( g^{\prime }\right) ^{n},
\label{identity-2-functions}
\end{equation}%
for $p=0,\ldots ,n$. Note that Eq. $\left( \ref{baran-identity}\right) $ is
the case $p=0$, $f_{1}=1$, $f_{2}=f$. Identity $\left( \ref%
{identity-2-functions}\right) $ can be extended to several functions (see
Theorem~\ref{theorem-1}).

We mention the similarity of Eq. $\left( \ref{identity-2-functions}\right) $
to a different identity for derivatives 
\begin{equation}
x\sum_{k=0}^{n}\binom{n}{k}\left( x^{k}f\left( x\right) \right) ^{\left(
k\right) }\left( x^{n-k}g\left( x\right) \right) ^{\left( n-k\right)
}=\left( x^{n+1}f\left( x\right) g\left( x\right) \right) ^{\left( n\right)
},  \label{id-linear-2-monomials}
\end{equation}%
which is a consequence of a generalization of the Leibniz Rule \cite%
{Abel-amm-2013} giving a closed form of the sum 
\begin{equation*}
\sum_{k=0}^{n}\binom{n}{k}\left( h^{k}f\right) ^{\left( k\right) }\left(
h^{n-k}g\right) ^{\left( n-k\right) }.
\end{equation*}

\section{The identity}

Our notation uses multi-indices. For $\mathbf{k}=\left( k_{1},\ldots
,k_{r}\right) \in \mathbb{Z}^{r}$, denote $\left\vert \mathbf{k}\right\vert
=k_{1}+\cdots +k_{r}$. For positive integers $n$, the binomial coefficient
is defined by $\binom{n}{\mathbf{k}}=\binom{n}{k_{1},\ldots ,k_{r}}:=\frac{n!%
}{k_{1}!\cdots k_{r}!\left( n-\left\vert k\right\vert \right) !}$.
Throughout the paper $\sum_{\left\vert \mathbf{k}\right\vert =n}$ means that
the summation runs over all $\mathbf{k}\in \left( \mathbb{Z}_{\geq 0}\right)
^{r}$ satisfying $\left\vert \mathbf{k}\right\vert =n$.

The following theorem contains the most general formula.

\begin{theorem}
\label{theorem-1}Let $n\in \mathbb{N}_{0}$, $r\in \mathbb{N}$, and let $%
f_{i} $, $g_{i}$ $\left( i=1,\ldots ,r\right) $ be functions which have a
derivative of order $n$. Suppose that $\sum_{i=1}^{r}g_{i}=0$. Then, for $%
\mathbf{s}=\left( s_{1},\ldots ,s_{r}\right) \in \left( \mathbb{Z}_{\geq
0}\right) ^{r}$, 
\begin{equation*}
\sum_{\left\vert \mathbf{k}\right\vert =n}\binom{n}{\mathbf{k}}%
\prod_{i=1}^{r}\left( f_{i}g_{i}^{k_{i}}\right) ^{\left( s_{i}\right)
}=\left\{ 
\begin{tabular}{lll}
$0$ &  & $\left( \left\vert s\right\vert <n\right) ,$ \\ 
&  &  \\ 
$n!\left( \prod_{i=1}^{r}f_{i}\right) \prod_{i=1}^{r}\left( g_{i}^{\prime
}\right) ^{s_{i}}$ &  & $\left( \left\vert s\right\vert =n\right) .$%
\end{tabular}%
\right.
\end{equation*}
\end{theorem}

\begin{proof}
The proof is based on the observation that, by Fa\`{a} di Bruno's formula, $%
\left( \frac{d}{dx}\right) ^{s}f^{n}\left( x\right) =0$, for $s=0,\ldots
,n-1 $, and $\left( \frac{d}{dx}\right) ^{n}f^{n}\left( x\right) =n!\left(
f^{\prime }\left( x\right) \right) ^{n}$ if $f\left( x\right) =0$. Let $%
f_{i} $, $g_{i}$ $\left( i=1,\ldots ,r\right) $ be functions which have a
derivative of order $n$ in a certain point $x\in \mathbb{R}$. Furthermore,
define $\mathbf{x}=\left( x_{1},\ldots ,x_{r}\right) \in \mathbb{R}^{r}$ and 
$\mathbf{1}=\left( 1,\ldots ,1\right) \in \mathbb{Z}^{r}$. We have 
\begin{eqnarray*}
&&\sum_{\left\vert \mathbf{k}\right\vert =n}\binom{n}{\mathbf{k}}%
\prod_{i=1}^{r}\left( f_{i}g_{i}^{k_{i}}\right) ^{\left( s_{i}\right)
}\left( x\right) \\
&=&\left. \left( \frac{\partial ^{\left\vert s\right\vert }}{\partial
x_{1}^{s_{1}}\cdots \partial x_{r}^{s_{r}}}\sum_{\left\vert \mathbf{k}%
\right\vert =n}\binom{n}{\mathbf{k}}\prod_{i=1}^{r}\left(
f_{i}g_{i}^{k_{i}}\right) \left( x_{i}\right) \right) \right\vert _{\mathbf{x%
}=\mathbf{1}x} \\
&=&\left. \left( \frac{\partial ^{\left\vert s\right\vert }}{\partial
x_{1}^{s_{1}}\cdots \partial x_{r}^{s_{r}}}\left[ \left(
\prod_{i=1}^{r}f_{i}\left( x_{i}\right) \right) \left(
\sum_{i=1}^{r}g_{i}\left( x_{i}\right) \right) ^{n}\right] \right)
\right\vert _{\mathbf{x}=\mathbf{1}x},
\end{eqnarray*}%
where we applied the binomial theorem. Note that $\sum_{i=1}^{r}g_{i}\left(
x\right) =0$ implies that 
\begin{equation*}
\left. \frac{\partial ^{\left\vert s\right\vert }}{\partial
x_{1}^{s_{1}}\cdots \partial x_{r}^{s_{r}}}\left( \sum_{i=1}^{r}g_{i}\left(
x_{i}\right) \right) ^{n}\right\vert _{\mathbf{x}=\mathbf{1}x}=\left\{ 
\begin{tabular}{lll}
$0$ &  & $\left( \left\vert s\right\vert <n\right) ,$ \\ 
&  &  \\ 
$n!\prod_{i=1}^{r}\left( g_{i}^{\prime }\left( x\right) \right) ^{s_{i}}$ & 
& $\left( \left\vert s\right\vert =n\right) .$%
\end{tabular}%
\right.
\end{equation*}%
Now the desired formula follows.
\end{proof}

Now we consider the particular case that the functions $g_{i}$ differ only
by constant factors, i.e., $g_{i}=c_{i}g$ with a function $g$ and $c_{i}\in 
\mathbb{R}$ $\left( i=1,\ldots ,r\right) $.

\begin{corollary}
\label{corollary2}Let $n\in \mathbb{N}_{0}$, $r\in \mathbb{N}$, and let $%
f_{i}$ $\left( i=1,\ldots ,r\right) $ and $g$ be functions which have a
derivative of order $n$. Suppose that $\mathbf{c}=\left( c_{1},\ldots
,c_{r}\right) \in \mathbb{R}^{r}$ satisfies $\left\vert \mathbf{c}%
\right\vert =0$. Let $\mathbf{s}=\left( s_{1},\ldots ,s_{r}\right) \in
\left( \mathbb{Z}_{\geq 0}\right) ^{r}$. Then 
\begin{equation*}
\sum_{\left\vert \mathbf{k}\right\vert =n}\binom{n}{\mathbf{k}}%
\prod_{i=1}^{r}c_{i}^{k_{i}}\left( f_{i}g^{k_{i}}\right) ^{\left(
s_{i}\right) }=\left\{ 
\begin{tabular}{lll}
$0$ &  & $\left( \left\vert \mathbf{s}\right\vert <n\right) ,$ \\ 
&  &  \\ 
$n!\left( \prod_{i=1}^{r}c_{i}^{s_{i}}\right) \left(
\prod_{i=1}^{r}f_{i}\right) \left( g^{\prime }\right) ^{\left\vert
s\right\vert }$ &  & $\left( \left\vert \mathbf{s}\right\vert =n\right) .$%
\end{tabular}%
\right.
\end{equation*}
\end{corollary}

\begin{remark}
In the special case $r=2$, with $c_{1}=-1=-c_{2}$, we have 
\begin{equation*}
\sum_{k=0}^{n}\left( -1\right) ^{k}\binom{n}{k}\left( f_{1}g^{k}\right)
^{\left( s_{1}\right) }\left( f_{2}g^{n-k}\right) ^{\left( s_{2}\right)
}=\left\{ 
\begin{tabular}{lll}
$0$ &  & $\left( s_{1}+s_{2}<n\right) ,$ \\ 
&  &  \\ 
$\left( -1\right) ^{s_{1}}n!f_{1}f_{2}\left( g^{\prime }\right) ^{n}$ &  & $%
\left( s_{1}+s_{2}=n\right) .$%
\end{tabular}%
\right.
\end{equation*}
\end{remark}

We close with some direct consequences.

\begin{example}
\label{marker-example1}Put $f_{i}\left( x\right) =x^{\alpha _{i}}$, $%
g_{i}\left( x\right) =x^{\beta }$ $\left( i=1,\ldots ,r\right) $. Corollary~%
\ref{corollary2} yields, for $\mathbf{s}=\left( s_{1},\ldots ,s_{r}\right)
\in \left( \mathbb{Z}_{\geq 0}\right) ^{r}$ with $\left\vert s\right\vert =n$
and $\mathbf{c}=\left( c_{1},\ldots ,c_{r}\right) \in \mathbb{R}^{r}$ with $%
\left\vert \mathbf{c}\right\vert =0$, the identity 
\begin{equation}
\sum_{\left\vert \mathbf{k}\right\vert =n}\binom{n}{\mathbf{k}}%
\prod_{i=1}^{r}\left( c_{i}^{k_{i}}\binom{\alpha _{i}+k_{i}\beta }{s_{i}}%
\right) =\binom{n}{\mathbf{s}}\beta ^{n}\prod_{i=1}^{r}c_{i}^{s_{i}}.
\label{example1}
\end{equation}%
Identity $\left( \ref{example1}\right) $ can be found in Gould's collection
of binomial identities \cite[(6.41)]{Gould-Identities-1974}. In the special
case $r=2$ we obtain, for $s=0,\ldots ,n$, 
\begin{equation}
\sum_{k=0}^{n}\left( -1\right) ^{k}\binom{n}{k}\binom{\alpha _{1}+k\beta }{s}%
\binom{\alpha _{2}+\left( n-k\right) \beta }{n-s}=\left( -1\right) ^{s}%
\binom{n}{s}\beta ^{n}.  \label{example1-special}
\end{equation}%
Formula $\left( \ref{example1-special}\right) $ is an easy consequence of
the fact that $\sum_{k=0}^{n}\left( -1\right) ^{n-k}\binom{n}{k}k^{i}=0$ $%
\left( i=0,\ldots ,n-1\right) $, $\sum_{k=0}^{n}\left( -1\right) ^{n-k}%
\binom{n}{k}k^{n}=n!$ and $\binom{\alpha _{1}+k\beta }{s}\binom{\alpha
_{2}+\left( n-k\right) \beta }{n-s}$ is a polynomial in the variable $k$ of
degree $n$ with leading term $\left( -1\right) ^{n-s}\beta ^{n}k^{n}/\left(
s!\left( n-s\right) !\right) $.
\end{example}

\begin{example}
If $f_{i}\left( x\right) =e^{\alpha _{i}x}$, $g_{i}\left( x\right) =e^{\beta
x}$ $\left( i=1,\ldots ,r\right) $ we obtain: 
\begin{equation}
\sum_{\left\vert \mathbf{k}\right\vert =n}\binom{n}{\mathbf{k}}%
\prod_{i=1}^{r}\left( c_{i}^{k_{i}}\left( \alpha _{i}+k_{i}\beta \right)
^{s_{i}}\right) =\binom{n}{\mathbf{s}}\beta ^{n}\prod_{i=1}^{r}c_{i}^{s_{i}}.
\label{example2}
\end{equation}%
In the special case $r=2$ we obtain, for $s=0,\ldots ,n$, 
\begin{equation}
\sum_{k=0}^{n}\left( -1\right) ^{k}\binom{n}{k}\left( \alpha _{1}+k\beta
\right) ^{s}\left( \alpha _{2}+\left( n-k\right) \beta \right) ^{n-s}=\left(
-1\right) ^{s}n!\beta ^{n}.  \label{example2-special}
\end{equation}%
As Eq. $\left( \ref{example1-special}\right) $ in Ex. $\ref{marker-example1}$%
, formula $\left( \ref{example2-special}\right) $ is obvious because $\left(
\alpha _{1}+k\beta \right) ^{s}\left( \alpha _{2}+\left( n-k\right) \beta
\right) ^{n-s}$ is a polynomial in the variable $k$ of degree $n$ with
leading term $\left( -1\right) ^{n-s}\beta ^{n}k^{n}$.
\end{example}

\strut

\end{document}